\documentclass{article}

\usepackage{amssymb,amsmath,amscd,dsfont}
\usepackage[all,cmtip,line]{xy}
\newcommand{\wh}[1]{\ensuremath{\widehat{#1}}}
\newcommand{\Map}{\ensuremath{\operatorname{Map}}}
\newcommand{\from}{\ensuremath{\nobreak\colon\nobreak}}
\renewcommand{\to}{\ensuremath{\nobreak\rightarrow\nobreak}}
\newcommand{\tot}{\mathrm{Tot}}
\newcommand{\colim}{\mathrm{colim}}
\newcommand{\tops}{\mathbf{Top}}
\newcommand{\ktop}{\mathbf{kTop}}
\newcommand{\fktop}{\mathrm{k}}
\newenvironment{tabsection}{}{}

\usepackage{amssymb,amsmath,amscd}

\usepackage[amsmath,thmmarks]{ntheorem}
\theoremseparator{.}
\theoremsymbol{\rule{1ex}{1ex}}
\theorembodyfont{\upshape}
\newtheorem{definition}{Definition}[section]
\newtheorem{remark}[definition]{Remark}
\newtheorem*{proof}{Proof}
\theorembodyfont{\itshape}
\newtheorem*{nntheorem}{Theorem}
\theoremsymbol{}
\newtheorem{lemma}[definition]{Lemma}
\newtheorem{proposition}[definition]{Proposition}
\newtheorem{theorem}[definition]{Theorem}
\newtheorem{corollary}[definition]{Corollary}

\usepackage[unicode]{hyperref}

\hypersetup{
    urlcolor = red,
	colorlinks = true,
	linkcolor = blue,
	citecolor = blue,
    linktocpage = true,
	pdftitle = {Topological Group Cohomology with Loop Contractible Coefficients},
    pdfauthor = {Martin Fuchssteiner, Christoph Wockel},
    bookmarksopen = true,
    bookmarksopenlevel = 1,
	unicode = true,
    hypertexnames = false
  }

\begin{document}

\title{Topological Group Cohomology with Loop Contractible Coefficients}
\author{Martin Fuchssteiner\\{\small\texttt{martin@fuchssteiner.net}}\and{}Christoph Wockel\\{\small\texttt{christoph@wockel.eu}}}
\maketitle

\begin{abstract}
 We show that for topological groups and loop contractible coefficients
 the cohomology groups of continuous group cochains and of group
 cochains that are continuous on some identity neighbourhood are
 isomorphic. Moreover, we show a similar statement for compactly
 generated groups and Lie groups holds and apply our results to
 different concepts of group cohomology for finite-dimensional Lie
 groups.
\end{abstract}

\section*{Introduction}

\begin{tabsection}
 For a topological group $G$ and a continuous $G$-module $A$ exist various
 concepts of ``topological group cohomology'', i.e., cohomology groups taking
 the topological properties of $G$ appropriately into account
 \cite{Segal70Cohomology-of-topological-groups,Deligne74Theorie-de-Hodge.-III,Moore76Group-extensions-and-cohomology-for-locally-compact-groups.-III}.
 When approaching from concrete cocycle models, there exist the naive notion of
 the cohomology $H_c^{*} (G,A)$ of the complex $C_{c}^* (G,A)$ of continuous
 group cochains and the more ad-hoc notion of the cohomology $H_{lc}^{*} (G,A)$
 of the complex $C^{*}_{lc}(G,A)$ of group cochains which are continuous only
 on some identity neighbourhood. The inclusion
 $C_{c}^{*}(G,A)\hookrightarrow C_{lc}^{*}(G,A)$ induces a homomorphism
 $H_{c}^{*}(G,A)\to H_{lc}^{*}(G,A)$ in cohomology, comparing these two
 notions.
 
 The locally continuous cohomology arises when combining non-trivial algebraic
 and topological information. For instance, the universal covering of a locally
 contractible and connected topological group $K$ is a central extension
 \begin{equation*}
  \pi_{1}(K)\to \widetilde{K}\to K,
 \end{equation*}
 which gives rise to a universal class in $H_{lc}^{2}(K,\pi_{1}(K))$, but
 $\widetilde{K}$ cannot be described by a class in $H_c^{*} (K,\pi_{1}(K))$.
 This construction has an immediate generalisation to higher connected covers
 \cite[Example
 V.2]{WagemannWockel11A-Cocycle-Model-for-Topological-and-Lie-Group-Cohomology}.
 If $K$ is, in addition, a (possibly infinite-dimensional) Lie group, and $Z$
 is an abelian Lie group, then the analogous notion of locally smooth
 cohomology $H^{2}_{ls}(K,Z)$ arises naturally in the integration problem of
 central and abelian extensions of Lie algebras
 \cite{Neeb02Central-extensions-of-infinite-dimensional-Lie-groups,Neeb04Abelian-extensions-of-infinite-dimensional-Lie-groups}.
 
 If $G$ is connected, then the continuous group cohomology $H^2_c (G,A)$
 describes abelian Lie group extensions which admit continuous global sections,
 i.e. they are trivial principal $A$-bundles. The ``locally continuous'' group
 cohomology $H^2_{lc} (G,A)$ describes abelian Lie group extensions which only
 admit local continuous sections, i.e. they are locally trivial principal
 $A$-bundles. For paracompact $G$ and (loop) contractible\footnote{A
 topological group $A$ is called loop contractible if there exists a
 contraction $H:A \times I \rightarrow A$ such that for each $t \in $ the map
 $H(-,t):A \rightarrow A$ is an endomorphism of $A$.} $A$ the Lie group
 extensions described by $H^2_{ls} (G,A)$ also admit a continuous global
 section, hence the morphism $i^2 : H^2_c (G,A) \rightarrow H^2_{lc} (G,A)$ is
 an isomorphism in this case. We show that this is true in every degree and
 that the same also holds for a Lie group $G$, a smooth $G$-module $A$ and the
 inclusion $C_{s}^* (G,A) \hookrightarrow C_{ls}^* (G,A)$ of smooth into
 ``locally smooth'' cochains:
\end{tabsection}

\begin{nntheorem}
 If $G$ is a topological group and $A$ a continuous $G$-module, then the
 natural inclusion
 \begin{equation*}
  C_{c}^* (G,A) \hookrightarrow C_{lc}^* (G,A)
 \end{equation*}
 induces an isomorphism in cohomology provided $A$ is loop contractible
 (e.g. a  topological vector space). If $G$ is a Lie group whose finite
 products are smoothly paracompact and $A$ is a smooth $G$-module, then
 \begin{equation*}
  C_{s}^* (G,A) \hookrightarrow C_{ls}^* (G,A)
 \end{equation*}
 induces an isomorphism in cohomology if $A$ is smoothly loop contractible.
\end{nntheorem}

\begin{tabsection}
 We also show that similar result also holds for compactly generated
 topological groups and modules. This has as an important consequence that the
 ``locally continuous'' cohomology coincides with the other mentioned
 topological group cohomologies, see \cite[Theorem
 IV.5]{WagemannWockel11A-Cocycle-Model-for-Topological-and-Lie-Group-Cohomology},
 which then has consequences for the regularity properties of measurable cocycles for general locally compact groups, see the end of Section 3 in \cite{Austin11On-discontinuities-of-cocycles-in-cohomology-theories-for-topological-groups}.

 As an application of our main result we show that it has as a consequence that
 ``locally smooth'' and ``locally continuous'' cohomology for
 finite-dimensional Lie groups coincide under a mild technical assumption.
\end{tabsection}

\section{Continuous and ``Locally Continuous'' Cochains}

Let $G$ be a topological group and $A$ be a topological $G$-module\footnote{In 
contrast to Section \ref{sect:the_compactly_generated_case} we thus assume that 
the multiplication and module map $G\times G\to G$ and $G\times A\to A$ are 
continuous for the product topology.}. In this section
we introduce the complexes of continuous and ``locally continuous''
(group) cochains and a double complex that will do the work for us when
showing the main theorem. We start with the \emph{continuous standard
cochain complex} $A^*_{c} (G,A)$ and its subcomplex $C_{c}^* (G,A)$ of
continuous homogeneous group cochains.

The \emph{standard cochain complex} is the complex 
$A^* (G , A):=\Map (G^{*+1},A )$ with differential \begin{equation*}
d \from A^n (G,A) \to A^{n+1}(G,A),\quad d f(g_{0},...,g_{n+1})=
\sum_{i=0}^{n+1}(-1)^{i}f(g_{0},...,\wh{g_{i}},...,g_{n+1}) \, , 
\end{equation*} 
and is always exact. (It is obtained by applying $\Map (-,A)$ to the standard
resolution of $\mathbb{Z}G$ cf. \cite[Chapter I.5]{Brown82Cohomology-of-groups}).
The continuous cochains form a sub-complex $C (G^{*+1},A )$, which is also 
exact and which we denote by $A_{c}^{*} (G,A)$. Somewhere in between there 
exists 
a complex of ``locally continuous'' cochains, which we are going to define now:

\begin{definition}
For each identity neighbourhood $U$ of a topological group $G$ we define the open diagonal neighbourhood 
$\Gamma_U^0 :=G$ in $G$ and $\Gamma_{U}^{p}$ in $G^{p+1}$ by setting 
 \begin{equation*}
  \Gamma_{U}^{p}:=\{(g_{0},...,g_{p})\in G^{p+1} \mid \, 
\forall \; 0 \leq i,j \leq p : g_i^{-1} g_j \in U\} \, ,
\end{equation*}
for $p \geq 1$. For each topological $G$-module $A$ the sub-complex 
\begin{equation*}
 A_{lc}^* (G,A):=\{f\from G^{*+1}\to A\mid f\text{ is continuous on some }\Gamma_{U}\}
\end{equation*}
of the standard complex is called the 
\emph{complex of locally continuous cochains}.
\end{definition}

The topological group $G$ acts on the open neighbourhood $\Gamma_{U}^{p}$ of the 
diagonal of $G^{p+1}$ via the diagonal action 
(i.e. by $g.(g_{0},...,g_{p})=(gg_{0},...,gg_{p})$), 
and thus $G$ acts on the cochain groups $A^{p}(G,A)$, $A_{lc}^{p}(G,A)$ and 
$A_{c}^{p}(G,A)$ via
 \begin{equation*}
  (g.f)(g_{0},...,g_{p})=g.\big(f(g^{-1}.(g_{0},...,g_{p}))\big) \, 
 \end{equation*}
leaving the corresponding sub-complexes invariant. The $G$-fixed points of 
this action are the $G$-equivariant functions, which are also called
\emph{homogeneous group cochains}.
The action intertwines the differentials on 
the complexes. As a consequence the subgroups
\begin{equation*}
C^{*}(G,A):=A^{*} (G,A)^{G},\;
C^{*}_{lc}(G,A):=A^{*}_{lc} (G,A)^G\;\text{ and }\;
C^{*}_{c}(G,A):=A^{*}_c(G,A)^G
\end{equation*}
form sub 
complexes of $A^{*}(G,A)$. 
These are the complexes of homogeneous group cochains, locally continuous 
homogeneous group cochains and continuous homogeneous group cochains respectively. 
We denote the corresponding cohomology groups by 
$H^{*} (G , A)$, $H_{lc}^{*} (G,A)$ and $H^{*}_{c} (G,A)$.

Let $\mathcal{U}_1$ be the neighbourhood filter of the identity in $G$ and 
consider the abelian groups 
\begin{equation*}
  A_{lc}^{p,q} ( G , A ) := \left\{ f:G^{p+1} \times G^{q+1} \rightarrow A
    \mid \; \exists U \in \mathcal{U}_1 \, : f_{\mid G^{p+1} \times
      \Gamma_U^q} \; \text{is continuous} \right\}
\, .
\end{equation*}
The abelian groups $A_{lc}^{p,q} ( G , A )$ form a first quadrant double 
complex with vertical and horizontal differentials given by 
\begin{align*}
 d_{h}^{p,q}\from A_{lc}^{p,q}\to A_{lc}^{p+1,q}, 
& \quad d_{h}^{p,q}(f^{p,q})(x_{0},...,x_{p+1},\vec{y})
=\sum_{i=0}^{p+1}(-1)^{i}f^{p,q}(x_{0},...,\wh{x_{i}},...,x_{p+1},\vec{y})\\
 d_{v}^{p,q}\from A_{lc}^{p,q}\to A_{lc}^{p,q+1}, 
&\quad d_{v}^{p,q}(f^{p,q})(\vec{x},y_{0},...,y_{q+1})
= (-1)^p \sum_{i=0}^{q+1}(-1)^{i}f^{p,q}(\vec{x},y_{0},...,\wh{y_{i}},...,y_{q+1}).
\end{align*}
The subgroups $A^{p,q}_{c} ( G , A ):= C (G^{p+1} \times G^{q+1} , A)$ 
of the groups $A_{lc}^{p,q} (G, A )$ form a sub double complex. 
Furthermore the groups $A_{lc}^{p,q} (G , A )^G$ and 
$A^{p,q}_{c} (G , A )^G$ of equivariant locally continuous and equivariant continuous
cochains form a sub double complex of 
$A_{lc}^{*,*} (G , A )$ and $A^{*,*}_{c} (G, A )$ respectively.

The rows $A_{lc}^{*,q} ( G , A )^G$ of the double complex 
$A_{lc}^{*,*} ( G, A )^G$ of equivariant cochains can be augmented by the 
cochain groups $C_{lc}^q ( G , A)$ of locally continuous homogeneous group 
cochains and the columns $A_{lc}^{p,*} ( G , A )^G$ can be augmented by the 
cochain groups $C_c^q ( G , A)$ of continuous homogeneous group 
cochains (cf. the computations in \cite[Section 2]{F11b}):

\begin{equation*}
  \vcenter{
  \xymatrix{
\vdots & \vdots & \vdots & \vdots \\ 
C_{lc}^2 (G ,A) \ar[r] \ar[u]_{d} 
& A_{lc}^{0,2} (G ,A)^G \ar[r]^{d_{h}}\ar[u]_{d_{v}} 
& A_{lc}^{1,2} (G ,A)^G \ar[r]^{d_{h}}\ar[u]_{d_{v}} 
& A_{lc}^{2,2} (G ,A)^G \ar[r]^(.6){d_{h}}\ar[u]_{d_{v}} & \cdots \\
C_{lc}^1 (G,A) \ar[r] \ar[u]_{d} 
& A_{lc}^{0,1} (G ,A)^G \ar[r]^{d_{h}}\ar[u]_{d_{v}} 
& A_{lc}^{1,1} (G ,A)^G \ar[r]^{d_{h}}\ar[u]_{d_{v}} 
& A_{lc}^{2,1} (G ,A)^G \ar[r]^(.6){d_{h}}\ar[u]_{d_{v}} 
& \cdots \\
C_{lc}^0 (G ,A) \ar[r] \ar[u]_{d} 
& A_{lc}^{0,0} (G ,A)^G \ar[r]^{d_{h}}\ar[u]_{d_{v}} 
& A_{lc}^{1,0} (G ,A)^G \ar[r]^{d_{h}}\ar[u]_{d_{v}} 
& A_{lc}^{2,0} (G ,A)^G \ar[r]^(.6){d_{h}}\ar[u]_{d_{v}} & \cdots \\
& C_c^0 (G,A) \ar[r]^{d}\ar[u] 
& C_c^1 (G,A) \ar[r]^{d}\ar[u] 
& C_c^2 (G,A) \ar[r]^(.6){d} \ar[u] 
& \cdots
}}
\end{equation*}
We denote the total complex of this double complex by 
$\tot  A_{lc}^{*,*} (G ,A)^G$. 
The augmentations of the rows induces a morphisms 
$j_h^* : C_{lc}^* (G,A) \hookrightarrow \tot A_{lc}^{*,*} (G ,A)^G$ of cochain
complexes. 
Likewise, the augmentations of the columns induces a morphism 
$j_v^* : C_c^* (G,A) \hookrightarrow \tot A_{lc}^{*,*} (G ,A)^G$. 
On each row $A_{lc}^{*,q} ( G , A )^G$ one can define a contracting homotopy 
$h^*$ by setting
   \begin{eqnarray}
    h^p : A^{p,q}_{lc} (G , A)^G & \rightarrow & A^{p-1,q}_{lc} (G , A)^{G}\notag \\
h^p (f)  (x_0,\ldots,x_{p-1},\vec{y}) & = & (-1)^p f
(x_0,\ldots,x_{p-1},y_0, \vec{y}) \label{eqn:contracting_homotopy}\,
\end{eqnarray}
(cf. the computations in \cite[Lemmata 2.1, 2.3, 2.8]{F11b} and 
\cite[Proposition 14.3.3]{F10}). 
As a consequence we note:
\begin{lemma}\label{lem:horizontal_isomorphism}
  The inclusion 
$j_h^* : C_{lc}^* (G,A) \hookrightarrow \tot A_{lc}^{*,*} (G ,A)^G$ induces an 
isomorphism in cohomology.
\end{lemma}

Note that this contraction does not work in vertical direction since it
would violate the continuity assumptions on the elements of
$A^{p,q-1}_{lc}(G,A)^{G}$.

\section{Continuous and Locally Continuous Group Cohomology}
\label{sect:continuous_and_locally_continuous_group_cohomology}

In the following we will show that for loop contractible coefficients
the inclusion of the complex $C_c^* (G,A)$ of continuous group cochains into the 
complex $C_{lc}^* (G,A)$ induces an isomorphism 
$H_c^{*} (G,A) \cong H_{lc}^{*} (G,A)$. 

Recall that a topological group $A$ is called \emph{loop contractible} if it is
contractible by a homotopy $H\from [0,1]\times A\to A$ such that each
$H(t,\mathinner{\cdot})$ is a group homomorphism (cf.\ \cite[Section 5]{F10}). 
It is the key observation of 
this article that this requirement on $A$ allows to adapt the procedure from
\cite{F11b} to our setting. 

The following will rely on the row exactness of the double complexes 
$A_c^{*,*} (G , A )^G$ and $A_{lc}^{*,*} ( G , A )^G$ from the previous section
and the following observation:

\begin{proposition} \label{keyprop}
 If the augmented column complexes
 $A_{c}^{p}(G,A) \hookrightarrow A^{p,*}_{lc}(G ,A)$ are exact, then
 the augmented sub column complexes
 $C^p_c (G,A) \hookrightarrow A_{lc}^{p,*} (G,A)^G$ of equivariant
 cochains are exact as well.
\end{proposition}

\begin{proof}
For the sake of completeness we recall the proof given in 
\cite[Proposition 3.1]{F11b}. 
 Assume that the augmented column complexes
 $A_{c}^{p}(G,A) \hookrightarrow A^{p,*}_{lc}(G ,A)$ are exact. Each
 equivariant vertical cocycle $f_{eq}^{p,q} \in A_{lc}^{p,q} (G,A)^G$ is
 the vertical coboundary $d_v f^{p,q-1}$ of a (not necessary
 equivariant) cocycle $f^{p,q-1}$ in $A_{lc}^{p,q-1} (G,A)$.
Define an equivariant cochain $f_{eq}^{p,q-1}$ of bidegree $(p,q-1)$ via
 \begin{equation*}
  f_{eq}^{p,q-1} (\vec{x},\vec{y}):= 
  x_0 . f^{p,q-1} (x_0^{-1} . \vec{x}, x_0^{-1} . \vec{y})
 \end{equation*}
 We assert that the vertical coboundary $d_v  f_{eq}^{p,q-1}$ of this
 equivariant cochain is the equivariant vertical cocycle $f_{eq}^{p,q}$.
 Indeed, since the differential $d_v$ is equivariant, the vertical
 coboundary of $f_{eq}^{p,q-1}$ computes to
 \begin{eqnarray*}
  d_v f_{eq}^{p,q-1} (\vec{x},\vec{y}) & = & 
  x_0 . \left[ d_v f^{p,q-1} (x_0^{-1} . \vec{x} , x_0^{-1} . \vec{y}) \right]
  \\
& = & 
  x_0 . \left[f_{eq}^{p,q} (x_0^{-1} . \vec{x} , x_0^{-1} . \vec{y}) \right] = 
  f_{eq}^{p,q} (\vec{x} , \vec{y}) \, .
 \end{eqnarray*}
 Thus every equivariant vertical cocycle $f_{eq}^{p,q}$ is the vertical
 coboundary of an equivariant cochain $f_{eq}^{p,q-1}$ of bidegree
 $(p,q-1)$.
\end{proof}

\begin{corollary}
 If the augmented column complexes
 $A_{c}^{p}(G,A) \hookrightarrow A^{p,*}_{lc}(G ,A)$ are exact, then the
 inclusion
 $j_v^* : C_{c}^* (G,A) \hookrightarrow \tot A_{lc}^{*,*} (G ,A)^G$
 induces an isomorphism in cohomology and the
 cohomologies $H_{lc}^{p} (G,A)$, $H^{p}(\tot A_{lc}^{*,*} (G,A)^G)$ and
 $H_{c} (G,A)^{p}$ are isomorphic.
\end{corollary}
It remains to show that in this case the isomorphism 
$H_{c}^{p} (G,A) \cong H_{lc}^{p} (G,A)$ is actually induced by the inclusion 
$ i^* : C_c^* (G,A) \hookrightarrow C_{lc}^* (G,A)$. 
Here the proof of \cite[Proposition 14.3.8]{F10} carries over almost in 
verbatim, see also \cite[Proposition 2.4]{F11b}:

\begin{proposition} \label{siscohtolc}
The image $j_v^n (f)$ of a homogeneous continuous group $n$-cocycle $f$ on $G$ in 
$\tot A_{lc}^{*,*} (G,A)^G$ is cohomologous to the image $j_h^n i^n  (f)$ of 
the locally continuous homogeneous group $n$-cocycle $i^n (f)$ in 
$\tot A_{lc}^{*,*} (G,A)^G$. 
\end{proposition}

\begin{proof}
 Consider a continuous homogeneous group $n$-cocycle $f\from G^{n+1}\to A$ on
 $G$ and define for all $p+q=n-1$ equivariant cochains 
 $\psi^{p,q}\from G^{p+1}\times G^{q+1}\cong G^{n+1}\to A$ in
 $A^{p,q}_{lc}(G,A)$ via
 $\psi^{p,q} ( \vec{x},\vec{y})= (-1)^p f (\vec{x},\vec{y})$. The
 vertical coboundary of the cochain $\psi^{p,q}$ is given by
 \begin{eqnarray*}
  [d_v \psi^{p,q}] (\vec{x},y_0,\ldots,y_{q+1}) & = & (-1 )^p
  \sum (-1)^i f (\vec{x},y_0,\ldots,\hat{y}_i,\ldots,y_q) \\
  & = & - \sum (-1)^{p+1+i} f ( x_0,\ldots,\hat{x}_i, \ldots,
  x_p,\vec{y}) \\
  & = & [d_h \psi^{p-1,q+1}](x_{0},...,x_{p},\vec{y}).
 \end{eqnarray*}
 The anti-commutativity of the horizontal and the vertical differential
 ensures that the coboundary of the cochain
 $\sum_{p+q=n-1} (-1)^p \psi^{p,q}$
in the total complex 
 is the cochain $j_v^n (f) - j_h^n i^n (f)$. Thus the cocycles $j^n (f)$
 and $j_h^n i^n (f)$ are cohomologous in $\tot A_{lc}^{*,*} (G,A)^G$.
\end{proof}

\begin{corollary}
 The homomorphism
 \begin{equation*}
 H (j_h^{p})^{-1} H (j_v^{p}): H_c^{p} (G,A) \rightarrow H_{lc}^{p} (G,A)
 \end{equation*}
 is induced by the inclusion
 $C_c^* (G,A) \hookrightarrow C_{lc}^* (G,A)$.
\end{corollary}

Recalling the exactness condition on the augmented columns of the double
complex $A^{*,*}_{lc} (G,A)$ we have shown:

\begin{theorem}
If $G$ is a topological group, $A$ a topological $G$-module and the augmented columns 
$A_c^* (G,A) \hookrightarrow A^{p,*}_{lc} (G,A)$
of the double complex $A^{*,*}_{lc} (G,A)$ are exact, then the inclusion 
$C^*_c (G,A) \hookrightarrow C_{lc}^* (G,A)$ induces an isomorphism 
in cohomology.
\end{theorem}

We now turn to showing that for loop contractible coefficients $A$ the 
exactness requirement in the above theorem is always satisfied.

\begin{proposition}\label{prop:loop_contractible_coefficients}
 If $A$ is loop contractible, then the augmented column complexes
 $A_{c}^{p}(G,A) \hookrightarrow A^{p,*}_{lc}(G ,A)$ are exact
\end{proposition}

\begin{proof}
 We consider for any open identity neighbourhood $U$ in $G$ the open
 neighbourhoods $\mathfrak{U}_U [n]:=\bigcup g.U^{n+1}$ of the diagonal,
 which form a simplicial subspace of $G^{*+1}$. This allows us to
 consider the cosimplicial group
 \begin{equation*}
  A^{p,*} (G,\mathfrak{U}_U ; A) := 
  \left\{ f : G^{p+1} \times \mathfrak{U}_U [*] \rightarrow A 
  \mid \, \forall \vec{g} \in \mathfrak{U}_U [*] : f (\mathinner{-},\vec{g}) \in C (G^{p+1},A) 
  \right\}
 \end{equation*}
 and its cosimplicial sub group
 $A_c^{p,*} (G,\mathfrak{U}_U ; A):= C ( G^{p+1} \times \mathfrak{U}_U [*] ,A)$
 as well as the cochain complex associated to both. By switching
 arguments the cochain group $A^{p,q} (G,\mathfrak{U}_U ; A)$ can be
 identified with the group
 $f : \mathfrak{U}_U [q]\rightarrow C (G^{p+1},A)$ of
 $\mathfrak{U}$-local cochains.
 Taking the colimit over all identity neighbourhoods yields the
complex of Alexander-Spanier cochains
 \begin{equation*}
  \colim A^{p,*} (G,\mathfrak{U}_U;A) \cong A_{AS}^* (G, C (G^{p+1},A))
 \end{equation*}
 of $G$ with coefficients $C (G^{p+1},A)$. It has been shown in
 \cite[Lemma 3.12]{F11b}\footnote{In the cited manuscript the complex
 $A_{lc}^{p,q} (G,A)$ is denoted by $A_{cg}^{p,q} (G,A)$ and one has to
 consider $G$ acting on itself by left translation.} that the augmented
 complex $A^{p}_c (G,A) \hookrightarrow A_{lc}^{p,*} (G,A)$ is exact if
 and only if the inclusion of colimit complexes
 \begin{equation*}
  \colim A_c^{p,*} (G,\mathfrak{U}_U;A) \hookrightarrow \colim A^{p,*} (G,\mathfrak{U}_U;A),
 \end{equation*}
 where $U$ runs over all open identity neighbourhoods in $G$, induces an
 isomorphism in cohomology.

 The latter can be shown by adapting the construction in \cite{F11a},
 where it is shown that the inclusion
 $A_{AS,c}^* (G,A) \hookrightarrow A_{AS}^* (G,A)$ of the continuous
 into the abstract Alexander-Spanier complex induces an isomorphism in
 cohomology. (For paracompact spaces and vector space coefficients this
 is a well-known fact \cite[Corollary 2.10]{F11a}, see \cite[Corollary
 2.14]{F11a} for not necessarily paracompact $G$.) In the case of loop
 contractible coefficients we replace the Alexander-Spanier presheafs
 $A_c^q (-,A)$ and $A^q (-,A)$ in the proof of \cite[Corollary
 5.20]{F11a} by the presheafs $A_c^{p,q} (G,-;A)$ and $A^{p,q} (G,-;A)$
 given by
 \begin{equation*}
  A^{p,q}_{c} (G,U;A):=  C(G^{p+1} \times U^{q+1},A)
 \end{equation*}
 and
 \begin{equation*}
  A^{p,q} (G,U;A):=  \left\{ f : G^{p+1} \times U^{q+1} \rightarrow A
  \mid \, \forall \vec{g}' \in U^{q+1} : f (-,\vec{g}') \in C (G^{p+1},A) 
  \right\}.
 \end{equation*}
 The arguments leading to \cite[Corollary 5.20]{F11a} carry over to show
 that for each fixed $p$ the inclusion
 $A_{c}^{p,*}(G,\mathfrak{U};A)\hookrightarrow A^{p,*}(G,\mathfrak{U};A)$
 induces an isomorphism in cohomology. We thus obtain the desired isomorphism
 \begin{equation}\label{eqn:cohom_iso}
 H (\colim A_c^{p,*} (G,\mathfrak{U}_U;A)) \cong H (\colim A^{p,*} (X,\mathfrak{U}_U;A)).
 \end{equation}
\end{proof}

\begin{remark}
 In the case that $G$ is locally compact, we can obtain the isomorphism
 \eqref{eqn:cohom_iso} directly from \cite[Corollary 2.14]{F11a}, since
 then $A_{c}^{p,q}(G,U;A)\cong A_{c}^{q}(U,C(G^{p},A))$ and thus the inclusion
 \begin{equation*}
  A_{c}^{p,*}(G,\mathfrak{U};A)\cong A_{c}^{*}(\mathfrak{U},C(G^{p},A))\hookrightarrow A^{p,q}(\mathfrak{U},C(G^{p},A))\cong A^{p,q}(G,\mathfrak{U};A)
 \end{equation*}
 induces an isomorphism in cohomology.
\end{remark}

\begin{corollary}\label{cor1}
 If $G$ is a topological group
 and $A$ is a loop contractible topological $G$-module, then the inclusion
 $C^*_c (G,A) \hookrightarrow C_{lc}^* (G,A)$ induces an isomorphism
 in cohomology.
\end{corollary}

\section{The compactly generated case}
\label{sect:the_compactly_generated_case}

\begin{tabsection}
 There exists another version of locally continuous cochains if one works
 solely in the category $\ktop$ of $k$-spaces \cite[\S
 7.9]{tomDieck08Algebraic-topology}, \cite[\S 2.4]{Hovey99Model-categories}.
 First of all this leads to a different notion of topological group, where one
 requires the group multiplication $G\times G\to G$ to be continuous in the
 $k$-topology, which is in general finer than the
 product topology. Moreover, this also affects the notion of $G$-module, where
 one also requires the action map $G\times A\to A$ to be continuous with
 respect to the $k$-topology. We then call $G$ a
 $k$-group and $A$ a $k$-module for $G$.
\end{tabsection}

The `k-ification' of topological spaces is a functor 
$\fktop : \tops \rightarrow \ktop$. 
Working in $\ktop$ one can still define continuous and `locally continuous'  
cochains by using products in the category $\ktop$ instead of $\tops$ 
(cf. \cite[Section 4]{F11b}):
\begin{eqnarray*}
C_{c,k}^q ( G , A ) & := & C ( \fktop G^{q+1} , A)^G \\
  C_{lc,k}^q ( G , A ) & := & \left\{ f \in A^q (G,A)^G \mid \; 
\exists U \in \mathcal{U}_1 \, : \left.f\right|_{\fktop \Gamma_{U}^{q}} \; 
\text{is continuous} \right\} 
\end{eqnarray*}
These subcomplexes of the standard complex lead to generally different 
versions of cohomology groups $H^{n}_{c,k}(G,A)$ and $H^{n}_{lc,k}(G,A)$
respectively, and the inclusion $ C^{*}_{c,k}(G,A)\to C^{*}_{lc,k}(G,A)$
induces a homomorphism $H_{c,k}(G,A)\to H_{lc,k}(G,A)$.
Likewise (cf. \cite[Section 4]{F11b}) one can define another bicomplex
 \begin{equation*}
  A_{lc,k}^{p,q} ( G , A ) := \left\{ f:G^{p+1} \times G^{q+1} \rightarrow G
  \mid \, \exists U \in \mathcal{U}_1 : \left.f\right|_{\fktop G^{p+1}
    \times_k \fktop \Gamma_{U}^{q}} \; \text{is continuous} \right\} \, ,
 \end{equation*}
 where $\fktop G^{p+1} \times_k \fktop \Gamma_{U}^{q}$ is the product in
 $\ktop$. As observed in \cite[Section 4]{F11b}) the rows an columns of the 
double complex $A_{lc,k}^{p,q} ( G , A )^G$ can be augmented by the complexes 
$C_{lc,k}^* (G,A)$ and $C_{lc,k}^* (G,A)$ respectively. The proof of the preceding
two sections carry over to yield the following results.

\begin{lemma}(cf.\ Lemma \ref{lem:horizontal_isomorphism})
The inclusion 
$j_{h,k}^{*}: C_{lc,k}^{*} (G,A) \hookrightarrow \tot A_{lc,k}^{*,*} (G,A)^G$ 
induces an isomorphism in cohomology. 
\end{lemma}

\begin{proposition}(cf.\ Proposition \ref{keyprop})
If the augmented column complexes 
$A_{c,k}^p (G ,A) \hookrightarrow A_{lc,k}^{p,*}(G,A)$ are exact, then the 
augmented sub column complexes 
$C_{c,k}^p (X;V)^G \hookrightarrow A_{lc,k}^{p,*}(X;V)^G$ are exact as well.
\end{proposition}

\begin{proposition}(cf.\ Proposition \ref{siscohtolc}) 
The image $j_{k,v}^n (f)$ of a homogeneous continuous group $n$-cocycle $f$ on 
$G$ in $\tot A_{lc, k}^{*,*} (G,A)^G$ is cohomologous to the image 
$j_{h,k}^n i^n  (f)$ of the locally continuous homogeneous group $n$-cocycle 
$i^n (f)$ in $\tot A_{lc, k}^{*,*} (G,A)^G$.
\end{proposition}

\begin{corollary}
 The homomorphism
 \begin{equation*}
 H (j_{h,k}^{p})^{-1} H (j_{v,k}^{p}): H_{c,k}^{p} (G,A) \rightarrow H_{lc,k}^{p} (G,A)
 \end{equation*}
is induced by the inclusion
 $C_{c,k}^* (G,A) \hookrightarrow C_{lc,k}^* (G,A)$.
\end{corollary}

\begin{corollary}
 If the augmented column complexes 
$A_{c, k}^{p}(G,A) \hookrightarrow A^{p,*}_{lc, k}(G ,A)$ are exact, then the 
inclusion 
$C_{c, k}^* (G,A) \hookrightarrow C_{lc, k}^{*,*} (G ,A)$ 
induces an isomorphism in cohomology.
\end{corollary}

\begin{tabsection}
 Up to here the procedure is exactly the same as in the Section
 \ref{sect:continuous_and_locally_continuous_group_cohomology}. However,
 we have to restrict the setting slightly for the main result of this
 sections that we now restrict to.
\end{tabsection}

\begin{proposition}
 If $G$ is a $k$-group whose finite products are $k$-spaces for the
 product topology and $A$ is a loop contractible continuous $G$-module,
 then $C^*_{c,k} (G,A) \hookrightarrow C_{lc,k}^* (G,A)$ induces an
 isomorphism $H_{c,k}^{p} (G,A) \cong H_{lc,k}^{p} (G,A)$ in cohomology.
\end{proposition}

\begin{proof}
 By the previous corollary we have to show that the augmented column
 complexes $A_{c, k}^{p}(G,A) \hookrightarrow A^{p,*}_{lc, k}(G ,A)$ are
 exact. The assumption that each $G^{k}$, endowed with the product
 topology, is a $k$-space ensures that the open diagonal neighbourhoods
 $\mathfrak{U}[k]$ are cofinal in the directed system of all open
 diagonal neighbourhoods. With this observation the proof of Proposition
 \ref{prop:loop_contractible_coefficients} of the exactness of
 $A_{c, k}^{p}(G,A) \hookrightarrow A^{p,*}_{lc, k}(G ,A)$ carries over
 to the compactly generated case if one replaces the results from
 \cite{F11b} and \cite{F11a} accordingly.

 In more detail, the cosimplicial group $A^{p,*}(G,\mathfrak{U};A)$ has
 to be replaced with
 \begin{equation*}
  A_k^{p,*} (G,\mathfrak{U};A) := 
  \left\{ f : G^{p+1} \times \mathfrak{U}[*] \rightarrow A 
  \mid \, \forall \vec{g} \in\mathfrak{U}[*] : f (-,\vec{g}) \in C (\fktop G^{p+1},A) 
  \right\},
 \end{equation*}
 \cite[Lemma 3.12]{F11b} has to be replaced with \cite[Lemma 5.12]{F11b}
 and \cite[Corollary 5.20]{F11a} has to be replaced with \cite[Corollary
 6.20]{F11a}.
\end{proof}

\section{Smooth and Locally Smooth Group Cohomology}

\begin{tabsection}
 From now on we assume that $G$ and $A$ are Lie groups and that $A$ is a
 smooth $G$-module\footnote{Manifolds are understood in the general
 infinite dimensional calculus from
 \cite{BertramGlocknerNeeb04Differential-calculus-over-general-base-fields-and-rings}}.
 Replacing the notion of continuity in Section
 \ref{sect:continuous_and_locally_continuous_group_cohomology} by
 smoothness we obtain the complex $C_s^* (G,A)$ of smooth homogeneous
 group cochains and the complex $C_{ls}^* (G,A)$ of homogeneous group
 cochains which are smooth only on some identity neighbourhood in $G$.
 These augment the columns and rows of the double complex
 $A_{ls}^{*,*} (G,A)^G$ given by the $G$-invariants of the double complex
 \begin{equation*}
  A_{ls}^{p,q} ( G , A ) := \left\{ f:G^{p+1} \times G^{q+1} \rightarrow G
  \mid \; \exists U \in \mathcal{U}_1 \, : f_{\mid G^{p+1} \times
  \Gamma_U^q} \; \text{is smooth} \right\}
  \, ,
 \end{equation*}
 We denote the total complex by
 $\tot  A_{ls}^{*,*} (G ,A)^G$. The augmentations of the rows induces a
 morphisms
 $j_h^* : C_{ls}^* (G,A) \hookrightarrow \tot A_{ls}^{*,*} (G ,A)^G$ of
 cochain complexes. Likewise, the augmentations of the columns induces a
 morphism
 $j_v^* : C_s^* (G,A) \hookrightarrow \tot A_{ls}^{*,*} (G ,A)^G$. On
 each row $A_{ls}^{*,q} ( G , A )^G$ one can define the same contracting
 homotopy $h^*$ as in \eqref{eqn:contracting_homotopy} showing
\end{tabsection}

\begin{lemma}(cf.\ Lemma \ref{lem:horizontal_isomorphism})
The inclusion $j_h^* : C_{ls}^* (G,A)\hookrightarrow \tot A_{ls}^{*,*}(G,A)^G$ 
induces an isomorphism in cohomology.
\end{lemma}

\begin{tabsection}
 Using the same arguments as in Section
 \ref{sect:continuous_and_locally_continuous_group_cohomology} we
 further obtain:
\end{tabsection}

\begin{proposition}(cf.\ Proposition \ref{keyprop})
 If the augmented complexes 
 $A_{s}^{p}(G,A) \hookrightarrow A^{p,*}_{ls}(G ,A)$ are exact, then the 
augmented sub column complexes
$C^p_s (G,A) \hookrightarrow A_{ls}^{p,*} (G,A)^G$ of equivariant cochains 
are exact as well.
\end{proposition}

\begin{proposition}(cf.\ Proposition \ref{siscohtolc})
The image $j_v^n (f)$ of a homogeneous smooth group $n$-cocycle $f$ on 
$G$ in $\tot A_{ls}^{*,*} (G,A)^G$ is cohomologous to the image 
$j_h^n i^n  (f)$ of the locally smooth homogeneous group $n$-cocycle 
$i^n (f)$ in $\tot A_{ls}^{*,*} (G,A)^G$.
\end{proposition}

\begin{corollary}
 The homomorphism 
 $H (j_h^{p})^{-1} H (j_v^{p}): H_s^{p} (G,A) \rightarrow H_{ls}^{p} (G,A)$ 
 is induced by the inclusion 
 $C_s^* (G,A) \hookrightarrow C_{ls}^* (G,A)$.
\end{corollary}

\begin{corollary}
 If the augmented complexes
 $A_{s}^{p}(G,A) \hookrightarrow A^{p,*}_{ls}(G ,A)$ are exact, then the
 inclusion $C_s^* (G,A) \hookrightarrow C_{ls}^* (G,A)$ induces an
 isomorphism in cohomology.
\end{corollary}

\begin{tabsection}
 Like in the previous section, the procedure was exactly the same as in
 Section \ref{sect:continuous_and_locally_continuous_group_cohomology},
 but now comes the point where we have to impose an additional condition
 on $G$.
\end{tabsection}

\begin{proposition}\label{prop2}
 If $G$ is a Lie group whose finite products are smoothly paracompact
 and $A$ is a smoothly loop contractible smooth $G$-module, then the
 inclusion $C^*_s (G,A) \hookrightarrow C_{ls}^* (G,A)$ induces an
 isomorphism $H_s^{p} (G,A) \cong H_{ls}^{p} (G,A)$ in cohomology.
\end{proposition}

\begin{proof}
 Analogously to Proposition \ref{prop:loop_contractible_coefficients}
 one shows that the augmented complexes
 $A_{s}^{p}(G,A) \hookrightarrow A^{p,*}_{ls}(G ,A)$ are exact. The
 assumption on each $G^{k}$ to be smoothly paracompact allows us to
 replace the results from \cite{F11b} and \cite{F11a} accordingly.

 In more detail, we replace the cosimplicial group
 $A^{p,*}(G,\mathfrak{U}_{U};A)$ with
 \begin{equation*}
  A^{p,*}_{s} (G,\mathfrak{U}_U ; A) := 
  \left\{ f : G^{p+1} \times \mathfrak{U}_U [*] \rightarrow A 
  \mid \, \forall \vec{g} \in \mathfrak{U}_U [*] : f (\mathinner{-},\vec{g}) \in C^{\infty} (G^{p+1},A) 
  \right\},
 \end{equation*}
 \cite[Lemma 3.12]{F11b} has to be replaced with \cite[7.12]{F11b} and
 \cite[Corollary 5.20]{F11a} has to be replaced with \cite[Theorem
 7.16]{F11a}.
\end{proof}

\section{Application to finite-dimensional Lie groups}

\begin{tabsection}
 In this section we show which impact the results of the previous
 sections have for finite-dimensional Lie groups. We first recall the
 following fact from \cite[Lemma
 IX.5.2]{BorelWallach00Continuous-cohomology-discrete-subgroups-and-representations-of-reductive-groups}
 or \cite[Theorem 5.1]{HochschildMostow62Cohomology-of-Lie-groups}.
\end{tabsection}

\begin{theorem}
 If $G$ is finite-dimensional, $\mathfrak{a}$ is a quasi complete
 locally convex space\footnote{A locally convex space is said to be
 quasi-complete if each bounded Cauchy net converges.} and a smooth
 $G$-module, then the inclusion
 $C^{*}_{s}(G,\mathfrak{a})\hookrightarrow C^{*}_{c}(G,\mathfrak{a})$
 induces an isomorphism in cohomology.
\end{theorem}

\begin{lemma}
 If $\Gamma$ is a discrete $G$-module, then the inclusion
 $C^{*}_{s}(G,\Gamma)\hookrightarrow C^{*}_{c}(G,\Gamma)$ is an
 isomorphism. In particular, it induces an isomorphism in cohomology.
\end{lemma}

\begin{proof}
 If $A$ is discrete, then smooth maps are the same thing as continuous
 maps.
\end{proof}

\begin{corollary}
 If $G$ is finite-dimensional, $\mathfrak{a}$ is a quasi complete
 locally convex space and a smooth $G$-module and
 $\Gamma\subset \mathfrak{a}$ is a discrete submodule, then the
 inclusion $C^{*}_{ls}(G,A)\hookrightarrow C^{*}_{lc}(G,A)$ induces an
 isomorphism in cohomology, where $A$ denotes the smooth $G$-module
 $\mathfrak{a}/\Gamma$.
\end{corollary}

\begin{proof}
 The exact sequence $\Gamma \hookrightarrow \mathfrak{a} \to A$ of
 coefficients admits a smooth local section and thus induces long exact
 sequences in locally smooth and locally continuous cohomology (the
 argument of \cite[Appendix
 E]{Neeb04Abelian-extensions-of-infinite-dimensional-Lie-groups} carries
 over literally to locally continuous group cohomology it the
 coefficient sequence admits a section which is continuous on some
 identity neighbourhood.). Together with the inclusions
 $C^{*}_{ls}(G,A)\hookrightarrow C^{*}_{lc}(G,A)$ this
 gives rise to the commuting diagram
 \begin{equation*}
  \xymatrix@=1em{
  H^{n}_{ls}(G,\Gamma)\ar[r]\ar[d]& H^{n}_{ls}(G,\mathfrak{a}) \ar[r]\ar[d]& H^{n}_{ls}(G,A)\ar[r]\ar[d]&H^{n+1}_{ls}(G,\Gamma)\ar[r]\ar[d]& H^{n+1}_{ls}(G,\mathfrak{a})\ar[d]\\
  H^{n}_{lc}(G,\Gamma)\ar[r]& H^{n}_{lc}(G,\mathfrak{a}) \ar[r]& H^{n}_{lc}(G,A)\ar[r]&H^{n+1}_{lc}(G,\Gamma)\ar[r]& H^{n+1}_{lc}(G,\mathfrak{a}) 
  }
 \end{equation*}
 with exact rows. The first and last two vertical morphisms are isomorphisms by
 the preceding results of this section and Corollary \ref{cor1} and Proposition
 \ref{prop2}. Thus the middle one is also an isomorphism by the five lemma.
\end{proof}

\begin{tabsection}
 The previous result does not hold in for infinite-dimensional $G$, as
 the case $G=C(S^{1},K)$ for $K$ a compact, simple and simply connected
 Lie group shows. The inclusion
 $C^{\infty}(S^{1},K)\hookrightarrow C(S^{1},K)$ is a homotopy
 equivalence by \cite[Remark
 A.3.8]{Neeb02Central-extensions-of-infinite-dimensional-Lie-groups} and
 thus the universal central extension of $C^{\infty}(S^{1},K)$ induces a
 topological non-trivial bundle
 \begin{equation*}
  U(1)\to P\to C(S^{1},K).
 \end{equation*}
 Now $P$ can be equipped with the structure of a topological group,
 since this is invariant under homotopy equivalences \cite[Prop.
 VII.1.3.5]{Grothendieck72}. This gives rise to a non-trivial element in
 $H^{2}_{lc}(C(S^{1},K),U(1))$. However, $C(S^{1},K)$ is simply
 connected and thus does not have any non-trivial central Lie group
 extension by \cite[Theorem
 7.12]{Neeb02Central-extensions-of-infinite-dimensional-Lie-groups} and
 Corollary 13 and Theorem 16 from
 \cite{Maier02Central-extensions-of-topological-current-algebras}. Thus
 $H^{2}_{ls}(C(S^{1},K),U(1))$ vanishes. However, $C(S^{1},K)$ is not
 smoothly paracompact, which would be the natural framework under which
 one would expect that $H^{*}_{ls}(G,A)$ is isomorphic to
 $H^{*}_{lc}(G,A)$. To the knowledge of the authors it is an open
 problem, whether for smoothly paracompact Lie groups $G$ and smooth
 $G$-modules $A=\mathfrak{a}/\Gamma$ the cohomologies 
$H_{lc}^* (G,A)$ and $H_{ls}^* (G,A)$ coincide or not. 
\end{tabsection}


\begin{thebibliography}{BGN04}
\providecommand{\url}[1]{\texttt{#1}}
\providecommand{\urlprefix}{URL }
\providecommand{\eprint}[2][]{\url{#2}}

\bibitem[Aus11]{Austin11On-discontinuities-of-cocycles-in-cohomology-theories-for-topological-groups} Austin, T.
\newblock \emph{{O}n discontinuities of cocycles in cohomology theories for topological groups}.
\newblock arXiv:1112.1465, 2011

\bibitem[BGN04]{BertramGlocknerNeeb04Differential-calculus-over-general-base-fields-and-rings}
Bertram, W., Gl{\"o}ckner, H. and Neeb, K.-H.
\newblock \emph{Differential calculus over general base fields and rings}.
\newblock Expo. Math. \textbf{22} (2004)(3):213--282

\bibitem[BW00]{BorelWallach00Continuous-cohomology-discrete-subgroups-and-representations-of-reductive-groups}
Borel, A. and Wallach, N.
\newblock \emph{Continuous cohomology, discrete subgroups, and representations
  of reductive groups}, \emph{Mathematical Surveys and Monographs}, vol.~67
  (American Mathematical Society, Providence, RI, 2000), second edn.

\bibitem[Br82]{Brown82Cohomology-of-groups}
Brown, Kenneth S. 
\newblock \emph{Cohomology of Groups}, \emph{Graduate Texts in Mathematics}, vol.~87
  (Springer-Verlag New York Inc., New York, NY, 10010)

\bibitem[Del74]{Deligne74Theorie-de-Hodge.-III}
Deligne, P.
\newblock \emph{Th{\'e}orie de {H}odge. {III}}.
\newblock Inst. Hautes {\'E}tudes Sci. Publ. Math. \textbf{44} (1974):5--77

\bibitem[Fuc10]{F10}
Fuchssteiner, M.
\newblock \emph{Transformation Groups and (Co)Homology}.
\newblock Ph.D. thesis, Technische Universit{\"a}t Darmstadt (Germany) 2010.
\newblock ISBN 978-3-8325-2524-8

\bibitem[Fuc11a]{F11a}
Fuchssteiner, M.
\newblock \emph{Cohomology of local cochains}.
\newblock arXiv:1110.2661, 2011

\bibitem[Fuc11b]{F11b}
Fuchssteiner, M.
\newblock \emph{A spectral sequence connecting continuous with locally
  continuous group cohomology}.
\newblock arXiv:1110.0994, 2011

\bibitem[Gr72]{Grothendieck72}
Grothendieck, A., ``{Seminaire de g{\'e}om{\'e}trie alg{\'e}brique Du
  Bois-Marie 1967-1969 (SGA 7 I)}.,'' {Lecture Notes in Mathematics. 288.
  Springer-Verlag Berlin-Heidelberg-New York}, 1972.

\bibitem[Hov99]{Hovey99Model-categories}
Hovey, M.
\newblock \emph{Model categories}.
\newblock \emph{Mathematical Surveys and Monographs},
  vol.~63 (American Mathematical Society, Providence, RI, 1999)

\bibitem[HM62]{HochschildMostow62Cohomology-of-Lie-groups}
Hochschild, G. and Mostow, G.~D.
\newblock \emph{Cohomology of {L}ie groups}.
\newblock Illinois J. Math. \textbf{6} (1962):367--401

\bibitem[Mai02]{Maier02Central-extensions-of-topological-current-algebras}
Maier, P.
\newblock \emph{Central extensions of topological current algebras}.
\newblock In \emph{Geometry and analysis on finite- and infinite-dimensional
  {L}ie groups (B\polhk edlewo, 2000)}, \emph{Banach Center Publ.}, vol.~55,
  pp. 61--76 (Polish Acad. Sci., Warsaw, 2002)

\bibitem[Moo76]{Moore76Group-extensions-and-cohomology-for-locally-compact-groups.-III}
Moore, C.~C.
\newblock \emph{Group extensions and cohomology for locally compact groups.
  {III}}.
\newblock Trans. Amer. Math. Soc. \textbf{221} (1976)(1):1--33

\bibitem[Nee02]{Neeb02Central-extensions-of-infinite-dimensional-Lie-groups}
Neeb, K.-H.
\newblock \emph{Central extensions of infinite-dimensional {L}ie groups}.
\newblock Ann. Inst. Fourier (Grenoble) \textbf{52} (2002)(5):1365--1442

\bibitem[Nee04]{Neeb04Abelian-extensions-of-infinite-dimensional-Lie-groups}
Neeb, K.-H.
\newblock \emph{Abelian extensions of infinite-dimensional {L}ie groups}.
\newblock In \emph{Travaux math\'ematiques. {F}asc. {XV}}, Trav. Math., XV, pp.
  69--194 (Univ. Luxemb., Luxembourg, 2004)

\bibitem[Seg70]{Segal70Cohomology-of-topological-groups}
Segal, G.
\newblock \emph{Cohomology of topological groups}.
\newblock In \emph{Symposia {M}athematica, {V}ol. {IV} ({INDAM}, {R}ome,
  1968/69)}, pp. 377--387 (Academic Press, London, 1970).

\bibitem[tD08]{tomDieck08Algebraic-topology}
tom Dieck, T.
\newblock \emph{Algebraic topology}.
\newblock EMS Textbooks in Mathematics (European Mathematical Society (EMS),
  Z{\"u}rich, 2008)

\bibitem[WW11]{WagemannWockel11A-Cocycle-Model-for-Topological-and-Lie-Group-Cohomology}
Wagemann, F. and Wockel, C.
\newblock \emph{A Cocycle Model for Topological and Lie Group Cohomology}.
\newblock arXiv:1110.3304, 2011

\end{thebibliography}

\def\polhk#1{\setbox0=\hbox{#1}{\ooalign{\hidewidth
  \lower1.5ex\hbox{`}\hidewidth\crcr\unhbox0}}} \def\cprime{$'$}

\end{document}